\documentclass[10pt,a4paper]{article}
\usepackage{fullpage}
\usepackage{amsfonts,amsmath,amssymb}
\usepackage{amsthm}
\usepackage{graphicx}\usepackage{sectsty}

\theoremstyle{plain}
\newtheorem{theorem}{Theorem}[section]
\newtheorem{lemma}[theorem]{Lemma}
\newtheorem{proposition}[theorem]{Proposition}

\theoremstyle{definition}
\newtheorem{definition}[theorem]{Definition}
\theoremstyle{remark}

\numberwithin{equation}{section}

\sectionfont{\large}
\newenvironment{acknowledgement}[1][Acknowledgement
]{\begin{trivlist} \item[\hskip \labelsep {\bfseries
#1}]}{\end{trivlist}}

\begin{document}
\title{An Inverse Uniqueness in  Interior Transmission
Problem\\ and Its Eigenvalue Tunneling in  Penetrable Simple Domains}
\author{Lung-Hui Chen$^1$}\maketitle\footnotetext[1]{Department of
Mathematics, National Chung Cheng University, 168 University Rd.
Min-Hsiung, Chia-Yi County 621, Taiwan. Email:
mr.lunghuichen@gmail.com;\,lhchen@math.ccu.edu.tw. Fax:
886-5-2720497.}
\begin{abstract}
We study an inverse uniqueness with a knowledge of spectral data in the interior transmission problem defined by an index of refraction in a simple domain.
We  expand the solution in such a domain into a series of one dimensional problems. For each one dimensional problem, we apply a value distribution theory in complex analysis to describe the eigenvalues of the system. By the orthogonality of the one dimensional system, we consider the uniqueness on the perturbation along each given incident angle.
\end{abstract}
\section{Introduction and Preliminaries}
In this paper, we study the inverse spectral problem in the following homogeneous interior
transmission problem
\begin{eqnarray}\label{1.1}
\left\{%
\begin{array}{ll}
    \Delta w+k^2n(x)w=0,  & \hbox{ in }D; \vspace{3pt}\\\vspace{3pt}
    \Delta v+k^2v=0, & \hbox{ in }D; \\\vspace{3pt}
    w=v, & \hbox{ on }\partial D; \\\vspace{3pt}
    \frac{\partial w}{\partial \nu}=\frac{\partial v}{\partial \nu},& \hbox{ on }\partial D,\\
\end{array}%
\right.
\end{eqnarray}
where $\nu$ is the unit outer normal; $D$ is a simple domain in $\mathbb{R}^3$ containing the origin with $\mathcal{C}^2$-boundary $\partial D$; $n(0)=1$,
$n(x)\in\mathcal{C}^2(\mathbb{R}^3)$; $n(x)>0$, for $x\in D$; $n(x)=1$, for $x\notin D$.  The equation~(\ref{1.1}) is called the homogeneous interior transmission eigenvalue problem.
We say $k\in\mathbb{C}$ is an interior transmission eigenvalue of~(\ref{1.1}) if there is a nontrivial pair of solution $(w,v)$ such that $w,v\in L^2(D)$,
$w-v\in H^2_0(D)$.

\par
The problem~(\ref{1.1}) occurs naturally when one considers the scattering of the plane waves by certain inhomogeneity inside the domain $D$, defined by an index of refraction in many models.
The interior transmission eigenvalues play a role in the inverse
scattering theory both in numerical computation and in theoretical
scattering theory. For the origin of interior transmission eigenvalue problem, we refer to Colton and Monk \cite{Colton}, and Kirsch \cite{Kirsch86}. For theoretical study and historic literature, we refer to \cite{Colton3,Colton2,Liu,Rynne}. The one-to-one correspondence between the  radiating solution of the Helmholtz equation and its far field patterns is well-known, but the denseness of the far field patterns remains further research topics.  It is also another subject of research interest to study the existence or location of the eigenvalues \cite{Cakoni,Cakoni2,Chen,Colton,Colton2,Kirsch,L,La,Mc}. It is expected to find a Weyl's type of asymptotics for the interior transmission eigenvalues. In that case, the distribution of the  eigenvalues is directly connected to certain invariant characteristics on the scatterer. In this regard, we apply the methods from entire function theory \cite{Cartwright, Cartwright2,Koosis, Levin,Levin2} to study the distributional laws of the eigenvalues. We also refer to \cite{Sun} for the reconstruction of the interior transmission eigenvalues, and \cite{Sun2} for a numerical description on  the distribution of the eigenvalues.  For the non-symmetrically-stratified medium, there are not too many known results \cite{Cakoni2,Colton3}. In this paper, we mainly follow the methods in \cite{Aktosun,Aktosun2,Chen,Chen3,Chen5} to study the non-symmetrical scatterers as a series of one-dimensional problems along the rays starting from the origin. The analysis along each ray possibly has multiple intersection points with $\partial D$, so we expect certain tunneling effect in a penetrable domain. In this paper, the new perspective is the following asymptotic analysis inside and outside the perturbation.

\par
We apply Rellich's expansion in scattering theory. Firstly we expand the solution $(v,w)$ of~(\ref{1.1}) in two series of spherical harmonics by Rellich's lemma \cite[p.\,32,]{Colton2}:
\begin{eqnarray}\label{1.3}
&&v(x;k)=\sum_{l=0}^{\infty}\sum_{m=-l}^{m=l}a_{l,m}j_l(k r)Y_l^m(\hat{x});\\
&&w(x;k)=\frac{1}{r}\sum_{l=0}^{\infty}\sum_{m=-l}^{m=l}b_{l,m}y_l(r)Y_l^m(\hat{x}), \label{1.4}
\end{eqnarray}
where $r:=|x|$, $0\leq r<\infty;$ $\hat{x}=(\theta,\varphi)\in\mathbb{S}^2$;
$j_l$ is the spherical Bessel function of first kind of order $l$. The summations converge uniformly and absolutely on suitable compact subsets in $|r|\geq R_0$, with some sufficiently large $R_0>0$. We note that the expansion~(\ref{1.3}) and~(\ref{1.4}) hold for Helmholtz equation without radial symmetry assumption on the perturbation \cite[p.\,31,\,Lemma\, 2.11]{Colton2}.
The uniqueness of the expansions outside $D$ are given by Rellich's lemma \cite[p.\,32]{Colton2}.
Particularly, the spherical harmonics
\begin{equation}\label{S}
Y_l^m(\theta,\varphi):=\sqrt{\frac{2l+1}{4\pi}\frac{(l-|m|)!}{(l+|m|)!}}P_l^{|m|}(\cos\theta)e^{im\varphi},
\,m=-l,\ldots,l;\,l=0,1,2,\ldots,
\end{equation}
form a complete orthonormal system in $\mathcal{L}^2(\mathbb{S}^2)$. Here,
\begin{equation}
P_n^m(t):=(1-t^2)^{m/2}\frac{d^mP_n(t)}{dt^m},\,m=0,1,\ldots,n,
\end{equation}
where the Legendre polynomials $P_n$, $n=0,1,\ldots,$ form a complete orthogonal system in $L^2[-1,1]$. We refer this to \cite[p.\,25]{Colton2}.
By the orthogonality of the spherical harmonics, the functions in the form
\begin{eqnarray}
\left\{
  \begin{array}{ll}
&v_{l,m}(x;k):=a_{l,m}j_l(k r)Y_l^m(\hat{x});
\vspace{8pt}\\\label{16}
&w_{l,m}(x;k):=\frac{b_{l,m}y_l(r;k)}{r}Y_l^m(\hat{x})
  \end{array}
\right.
\end{eqnarray}
satisfy the first two equations in~(\ref{1.1}) independently \cite[p.\,227]{Colton2} for $r\geq R_0$ .

\par
Now we consider $k\in\mathbb{C}$, $a_{l,m}$, and $b_{l,m}$, if any,  satisfies the following independent system for  $l\in\mathbb{N}_0$ and $-l\leq m\leq l$,
\begin{eqnarray}\vspace{5pt}\label{18}
\left\{
  \begin{array}{ll}
    &a_{l,m}j_l(k r)|_{r=R_0}=b_{l,m}\frac{y_l(r;k)}{r}|_{r=R_0}; \vspace{8pt}\\
    &a_{l,m}[j_l(k r)]'|_{r= R_0}=b_{l,m}[\frac{y_l(r;k)}{r}]'|_{r= R_0}.
  \end{array}
\right.
\end{eqnarray}
In terms of elementary linear algebra, the existence of the coefficients $a_{l,m}$ and $b_{l,m}$ are equivalent to finding the zeros of the following functional determinant:
\begin{eqnarray}\label{1.12}
D_{l}(k;R_0)&:=&\det\left(%
\begin{array}{cc}
  j_l(k r)|_{r=R_0}  & -\frac{y_l(r;k)}{r}|_{r=R_0}\vspace{6pt}\\
  \{j_l(k r)\}'|_{r=R_0}& -\{\frac{y_l(r;k)}{r}\}'|_{r=R_0}\\
\end{array}%
\right)\\\vspace{12pt}\label{110}
&=&-j_l(k R_0)\frac{{y}'_l(R_0;k)}{R_0}+j_l(k R_0)\frac{y_l(R_0;k)}{R_0^2}+k {j}'_l(R_0)\frac{y_l(R_0;k)}{R_0},
\end{eqnarray}
in which the system is independent of $m$ and $\hat{x}$. The forward problem  describes the distribution of the zeros of $D_{l}(k;R_0)$, while the inverse problem  specifies the index of refraction $n$ by the zero set. In \cite{Chen,Chen3,Chen5,Mc}, we have discussed the methods to find the zeros of $D_{l}(k;R_0)$.

\par
Let $k\in\mathbb{C}$ be a possible eigenvalue of~(\ref{18}). Applying the analytic continuation of the Helmholtz equation and Rellich's lemma \cite[p.\,32,\,33;\,p.\,222]{Colton2},  the solutions parameterized by $k$ solve
\begin{eqnarray}\label{113}
\left\{%
\begin{array}{ll}
    w(x;k)\equiv v(x;k),&x\notin D; \vspace{8pt}\\
\frac{\partial w(x;k)}{\partial \nu}\equiv \frac{\partial v(x;k)}{\partial \nu},&x\notin D,\\
\end{array}%
\right.
\end{eqnarray}
in the simple domain $D$.
 Conversely,~(\ref{113}) implies~(\ref{1.12}) by the uniqueness of the system~(\ref{16}) at $r=R_0$.

\par
We note that the representation~(\ref{1.3}) and~(\ref{1.4}) can merely hold outside $|x|\geq R_0$, so we seek to extend the representation into $|x|\leq R_0$.
Let $\hat{x}_1\in\mathbb{S}^2$ be a given incident direction satisfying the following geometric condition:
\begin{equation}
\mbox{the line segment from }(R_0,\hat{x}_1) \mbox{ to } (r_1,\hat{x}_1), \,r_1\leq R_0, \mbox{lies outside }D\mbox{ with }(r_1,\hat{x}_1)\in\partial D.
\end{equation}
For this $\hat{x}_1$, we extend each Fourier coefficient $y_l(r;k)$ with  $k\in\mathbb{C}$ determined by system~(\ref{18}) for all $l$, toward the origin until it meets the boundary $\partial D$ at $(r_1,\hat{x}_1)$. Along the given $\hat{x}_1$, we apply the differential operator $\Delta+k^2n$ with
\begin{equation}
\Delta=\frac{1}{r^2}\frac{\partial}{\partial r}r^2\frac{\partial}{\partial r}+\frac{1}{r^2\sin{\varphi}}\frac{\partial}{\partial \varphi}\sin\varphi\frac{\partial}{\partial \varphi}
+\frac{1}{r^2\sin^2{\varphi}}\frac{\partial^2}{\partial \theta^2}
\end{equation}
to $\{w_{l,m}(x)\}$, which accordingly can solve the problem~(\ref{1.1}) with the manmade index of refraction $n(x)=n(r\hat{x})=n(r\hat{x}_1)$ for all $\hat{x}\in\mathbb{S}^2$. More importantly, the analytic continuation~(\ref{113}) and the interior transmission condition imply the following ODE:
\begin{eqnarray}\label{118}
\left\{
  \begin{array}{ll}
    y_l''(r;k)+(k^2n(r\hat{x}_1)-\frac{l(l+1)}{r^2})y_l(r;k)=0;\vspace{8pt}\\
    D_{l}(k;r_1)=0.
  \end{array}
\right.
\end{eqnarray} 
\par
If there is merely one intersection point for $[0,R_0]\times \hat{x}_1$ with $\partial D$, then we set  the initial conditions to~(\ref{118})
\begin{equation}\label{119}
y_0(0;k)=0,\, y_0'(0;k)=1; \,y_l(0;k)=0,\, y_l'(0;k)=0,\,l\in\mathbb{N}.
\end{equation}
That is,
\begin{eqnarray}
D_{l}(k;0)=0.\label{1.18}
\end{eqnarray}
We refer to the details to \cite{Aktosun,Colton2,Mc}.
For $l\geq0$, we can take $a_{l,m}=b_{l,m}=1$ in~(\ref{18}) by the uniqueness implied by~(\ref{119}).
Now, the uniqueness of the ODE~(\ref{18}) is valid up to the boundary $\partial D$,
\begin{eqnarray}\label{115}\vspace{5pt}
\left\{
  \begin{array}{ll}
    &a_{l,m}j_l(k r)=b_{l,m}\frac{y_l(r;k)}{r}; \vspace{8pt}\\
    &a_{l,m}[j_l(k r)]'=b_{l,m}[\frac{y_l(r;k)}{r}]',\,r=r_1,\,\hat{x}_1\in\mathbb{S}^2.
  \end{array}
\right.
\end{eqnarray}
That is, $D_{l}(k;r_1)=D_{l}(k;R_0)$ by the uniqueness of ODE~(\ref{118}) along the line segment $(R_0,\hat{x}_1)$ to $(r_1,\hat{x}_1)$.

\par
In general, the solution $y_l(r)$ depends on the incident direction $\hat{x}$ whenever entering the perturbation, so we denote the extended solution of~(\ref{118}) as $\hat{y}_l(r;k)$, and the functional determinant as $\hat{D}_{l}(k;r_1)$. Thus,~(\ref{118}) is rephrased as
\begin{eqnarray}\label{1.10}
\left\{
  \begin{array}{ll}
    \hat{y}_l''(r;k)+(k^2n(r\hat{x})-\frac{l(l+1)}{r^2})\hat{y}_l(r;k)=0;\vspace{8pt}\\
    \hat{D}_{l}(k;0)=0,\,\hat{D}_{l}(k;r_1)=0.
  \end{array}
\right.
\end{eqnarray}
The eigenvalues of~(\ref{1.10}) are discussed in \cite{Aktosun,Chen,Chen3,Chen5} by the singular Sturm-Liouville theory in \cite{Carlson,Carlson2,Carlson3}.
\par
However, the domain $D$ is not starlike in general. Instead of~(\ref{115}), we now ask for any $k\in\mathbb{C}$ such that, the following system holds
\begin{eqnarray}\label{1.13}\vspace{5pt}
\left\{
  \begin{array}{ll}
    &\hat{a}_{l,m}j_l(k r)|_{r= \hat{r}}=\hat{b}_{l,m}\frac{\hat{y}_l(r;k)}{r}|_{r= \hat{r}}; \vspace{10pt}\\
    &\hat{a}_{l,m}[j_l(k r)]'|_{r= \hat{r}}=\hat{b}_{l,m}[\frac{\hat{y}_l(r;k)}{r}]'|_{r= \hat{r}},\, \hat{r}\in\hat{R},
  \end{array}
\right.
\end{eqnarray}
in which $\hat{R}$ is the intersection set along the incident angle $\hat{x}$ defined by
\begin{eqnarray}\label{boundary}
\hat{R}:=\{\hat{r}|\, (\hat{r},\hat{x})\in\partial D\}=\{\hat{r}_1,\ldots,\hat{r}_{\hat{M}}\},
\end{eqnarray}
and $\hat{y}_l(r;k)$ are defined by the solutions of first line of equation~(\ref{118}) with initial condition~(\ref{1.13}).

 In general, we assume $\hat{R}$ to be a finite discrete set and $\hat{r}_1<\hat{r}_2<\ldots<\hat{r}_{\hat{M}}$. In the case that $(\alpha,\hat{x})$, $(\beta,\hat{x})$, and $(\gamma,\hat{x})$ are any three consecutive points along the incident direction $\hat{x}$. Whenever $(\beta,\hat{x})$ is a tangent point at the boundary, we disregard it and consider  the line segment from $(\alpha,\hat{x})$ to $(\gamma,\hat{x})$ as  either completely inside or outside the perturbation. Without loss of generality, we assume that $\hat{R}$ contains no tangent point. See Fig.~\ref{simple}.

\par
Firstly starting with the first segment into the perturbation, we discuss the well-posedness of the initial value problem starting at $\hat{r}_{\hat{M}}$:
\begin{eqnarray}\label{1.22}
\left\{
  \begin{array}{ll}
    \hat{y}_l''(r)+(k^2n(r\hat{x})-\frac{l(l+1)}{r^2})\hat{y}_l(r)=0,\,\hat{r}_{\hat{M}-1}\leq r\leq\hat{r}_{\hat{M}};\vspace{8pt}\\
    \hat{D}_{l}(k;\hat{r}_{\hat{M}})=0,
  \end{array}
\right.
\end{eqnarray}
which has an unique solution inward  to $\hat{r}_{\hat{M}-1}$ given $j_l(k\hat{r})$ as a known function. The behavior of the solution is understood by the singular Sturm-Liouville theory in Section 2.  Because $\hat{r}_{\hat{M}}$ is the first intersection, $w_{l,m}(\hat{x}\hat{r}_{\hat{M}};k)=v_{l,m}(\hat{x}\hat{r}_{\hat{M}};k)$,  the uniqueness of~(\ref{1.22}) holds to the $\hat{r}_{\hat{M}-1}$ for  $k\in\mathbb{C}$. Because of~(\ref{113}) and the construction of~(\ref{16}), the unique analytic continuation of $w_{l,m}(x,k)$ holds outside $D$ as well as along the $\hat{x}$ inside $D$. Thus, $\hat{D}_{l}(k;\hat{r}_{\hat{M}-1})=0$ holds by analytic continuation of the Helmholtz equation. More importantly, it filters out a discrete set of eigenvalues of
\begin{eqnarray}\label{223}
\left\{
  \begin{array}{ll}
    \hat{y}_l''(r)+[k^2n(r\hat{x})-\frac{l(l+1)}{r^2}]\hat{y}_l(r)=0,\,\hat{r}_{\hat{M}-1}\leq r\leq\hat{r}_{\hat{M}};\vspace{8pt}\\
    \hat{D}_{l}(k;\hat{r}_{\hat{M}})=0,\,\hat{D}_{l}(k;\hat{r}_{\hat{M}-1})=0.
  \end{array}
\right.
\end{eqnarray}
Let $k^T$ be one of its eigenvalues. Leaving the perturbation at $\hat{r}_{\hat{M}-1}$, the  $k^T$ defines another ODE system:
\begin{eqnarray}\label{224}
\left\{
  \begin{array}{ll}
    \hat{y}_l''(r)+[k^2n(r\hat{x})-\frac{l(l+1)}{r^2}]\hat{y}_l(r)=0,\,\hat{r}_{\hat{M}-2}\leq r\leq\hat{r}_{\hat{M}-1};\vspace{8pt}\\
  \hat{D}_{l}(k;\hat{r}_{\hat{M}-1})=0,\,\hat{D}_{l}(k;\hat{r}_{\hat{M}-2})=0,
  \end{array}
\right.
\end{eqnarray}
in which $\hat{D}_{l}(k^T;\hat{r}_{\hat{M}-2})=0$ holds due to the analytic continuation of Helmholtz equation.
The same $k^T$ appears at $\hat{r}_{\hat{M}-2}$ and ready to define another new ODE.
\begin{eqnarray}\label{225}
\left\{
  \begin{array}{ll}
    \hat{y}_l''(r)+[k^2n(r\hat{x})-\frac{l(l+1)}{r^2}]\hat{y}_l(r)=0,\,\hat{r}_{\hat{M}-3}\leq r\leq\hat{r}_{\hat{M}-2};\vspace{8pt}\\
  \hat{D}_{l}(k;\hat{r}_{\hat{M}-2})=0,\,\hat{D}_{l}(k;\hat{r}_{\hat{M}-3})=0.
  \end{array}
\right.
\end{eqnarray}
By analytic continuation, $k^T$ satisfies~(\ref{223}),~(\ref{224}), and~(\ref{225}), and  appears at $\hat{r}_{\hat{M}-3}$ as well. More importantly, the system~(\ref{224}) produces new eigenvalues that appear at $\hat{r}_{\hat{M}-3}$ and
consecutively into each intersection interval. Thus, we have
\begin{eqnarray}\label{226}
\left\{
  \begin{array}{ll}
    \hat{y}_l''(r)+[k^2n(r\hat{x})-\frac{l(l+1)}{r^2}]\hat{y}_l(r)=0,\,\hat{r}_{0}:=0<r<\infty;\vspace{8pt}\\
    \hat{D}_{l}(k;\hat{r}_{0})=0,\, \hat{D}_{l}(k;\hat{r}_{1})=0,\,\ldots,\hat{D}_{l}(k;\hat{r}_{\hat{M}})=0.
  \end{array}
\right.
\end{eqnarray}
Each element of the zero set of $\hat{D}_{l}(k;\hat{r}_j)$ defines an initial value problem in the neighboring interval. There is an uniqueness to the solution of~(\ref{226}) by the piecewise construction as shown above, and we call the extended solution $\hat{y}(r;k)$ for each $k$ the eigenvalue tunneling in interior transmission problem. Such a construction can be set to initiate at $\hat{r}_{0}$ and tunnels to the infinity. We have already discussed the simple case as in the starlike domains \cite{Chen,Chen3,Chen5}:
\begin{eqnarray}\label{125}
\left\{
  \begin{array}{ll}
    \hat{y}_l''(r)+(k^2n(r\hat{x})-\frac{l(l+1)}{r^2})\hat{y}_l(r)=0,\,\hat{r}_{0}<r<\infty;\vspace{8pt}\\
    \hat{D}_{l}(k;\hat{r}_{0})=0,\, \hat{D}_{l}(k;\hat{r}_{1})=0.
  \end{array}
\right.
\end{eqnarray}
With the initial condition $\hat{D}_{l}(k;0)=0$, the function $\hat{D}_{l}(k;\hat{r}_{1})$ is an entire function of exponential type \cite{Carlson,Carlson2,Carlson3,Chen,Chen3,Chen5,Po}. The eigenvalues are its zeros. Without loss of generality, we take $\hat{r}_0$ as the reference point. Thus, the eigenvalues of~(\ref{125}) form a discrete set in $\mathbb{C}$, accumulate into the eigenvalues of~(\ref{226}), and tunnel to the infinity.

\par
Conversely, once we find an eigenvalue of~(\ref{226}) for some  $l$  along some $\hat{x}$, it solves~(\ref{18}) by the uniqueness of ODE and then~(\ref{113}) by the analytic continuation of Helmholtz equation.
Whenever we collect all such eigenvalues from each incident $\hat{x}\in\mathbb{S}^2$, they are interior transmission eigenvalues of~(\ref{1.1}) by analytic continuation. The geometric characteristics of the perturbation are connected by rays of ODE system to the far fields.

\begin{theorem}\label{13}
Let $n^1$, $n^2$ be two unknown indices of refraction as assumed in~(\ref{1.1}). If they have the same set of eigenvalues, then $n^1\equiv n^2$.
\end{theorem}


\section{Asymptotic Expansions and Cartwright-Levinson Theory}
To study the functional determinants $\hat{D}_l(k;\hat{r})$, we collect the following asymptotic behaviors of $\hat{y}_l(r;k)$ and $\hat{y}_l'(r;k)$. For $l=0$, we apply the
 Liouville transformation \cite{Carlson,Carlson2,Carlson3,Colton3,Colton2,Po}:
\begin{eqnarray}\label{L}
&\hat{z}_0(\hat{\xi}):=[n(r\hat{x})]^{\frac{1}{4}}\hat{y}_0(r;k),
\end{eqnarray}
where
\begin{equation}
\hat{B}(r):=\hat{\xi}(r)=  \int_0^r[n(\rho\hat{x})]^{\frac{1}{2}}d\rho,\,0\leq r\leq\hat{r}_1.
\end{equation}
Here we recall that $n$ is $1$ outside $D$.
For simplicity of the notation, we drop all the  superscripts about $\hat{x}$ whenever the context is clear.
\begin{definition}
Let $f(z)$ be an integral function of order $\rho$, and let
$N(f,\alpha,\beta,r)$ denote the number of the zeros of $f(z)$
inside the angle $[\alpha,\beta]$ and $|z|\leq r$. We define the
density function as
\begin{equation}\label{Den}
\Delta_f(\alpha,\beta):=\lim_{r\rightarrow\infty}\frac{N(f,\alpha,\beta,r)}{r^{\rho}},
\end{equation}
and
\begin{equation}
\Delta_f(\beta):=\Delta_f(\alpha_0,\beta),
\end{equation}
with some fixed $\alpha_0\notin E$ such that $E$ is at most a
countable set.
\end{definition}
\begin{lemma}\label{22}
The functional determinant $\hat{D}_l(k;r)$ is of order one and of type $r+\hat{B}(r),\,\hat{r}_0\leq r\leq\hat{r}_1$. In particular,
\begin{equation}\label{222}
\Delta_{\hat{D}_l(k;r)}(-\epsilon,\epsilon)=\frac{r+\hat{B}(r)}{\pi}.
\end{equation}
\end{lemma}
\begin{proof}
We begin with~(\ref{1.12}).
\begin{eqnarray}
\hat{D}_l(k;r)
&=&-j_l(k r)\frac{{\hat{y}}'_l(r;k)}{r}+j_l(k r)\frac{\hat{y}_l(r;k)}{r^2}+k {j}'_l(k r)\frac{\hat{y}_l(r;k)}{r}\nonumber\\\vspace{15pt}\nonumber
&=&\frac{k{j}'_l(k r)\hat{y}_l(r;k)}{r}\{1-\frac{1}{k}\frac{j_l(k r)}{{j}'_l(k r)}\frac{{\hat{y}}'_l(r;k)}{\hat{y}_l(r;k)}
+\frac{1}{k r}\frac{j_l(k r)}{{j}'_l(k r)}\}\\
&=&\frac{k{j}'_l(k r)\hat{y}_l(r;k)}{r}\{\hat{\alpha}_l(k)+O(\frac{1}{k})\},\label{216}
\end{eqnarray}
in which
\begin{equation}
\hat{\alpha}_l(k):=1-\frac{1}{k}\frac{j_l(k r)}{{j}'_l(k r)}\frac{{\hat{y}}'_l(r;k)}{\hat{y}_l(r;k)};
\end{equation}
We have $\frac{j_l(k\hat{r})}{{j}'_l(k r)}=O(1)$ outside the zeros of ${j}'_l(k r)$;
similarly, $\frac{{\hat{y}}'_l(r;k)}{\hat{y}_l(r;k)}=O(k)$ outside the zeros of $\hat{y}_l(r;k)$. The term $\hat{\alpha}(k)$ is bounded and bounded away from zero outside the zeros of ${j}'_l(k r)$ and $\hat{y}_l(r;k)$ on real axis.

\par
Consequently, we can compute the Lindel\"{o}f's indicator function \cite{Chen,Chen3,Chen5,Levin,Levin2} for $\hat{D}_l(k)$ for~(\ref{216}):
\begin{equation}\label{2.16}
h_{\hat{D}_l(k)}(\theta)=h_{{j}'_l(k\hat{r})}(\theta)+h_{\hat{y}_l(\hat{r};k)}(\theta)=(\hat{r}+\hat{B}(\hat{r}))|\sin\theta|,\,\theta\in[0,2\pi],\,\hat{r}_0\leq \hat{r}\leq\hat{r}_1,
\end{equation}
in which $\hat{B}(\hat{r})= \int_0^{\hat{r}}[n(\rho\hat{x})]^{\frac{1}{2}}d\rho$, and we apply an inequality of Lindel\"{o}f's indicator function in \cite[p.\,51]{Levin}.

In case that $\hat{\alpha}_l(k)\equiv0$, instead of~(\ref{2.16}), we have
\begin{equation}
h_{\hat{D}_l(k)}(\theta)=h_{{j}_l(k\hat{r})}(\theta)+h_{\hat{y}_l(\hat{r};k)}(\theta)=(\hat{r}+\hat{B}(\hat{r}))|\sin\theta|,\,\theta\in[0,2\pi],\,\hat{r}_0\leq \hat{r}\leq\hat{r}_1,
\end{equation}
Referring Cartwright theory to \cite[p.\,251]{Levin}, we conclude that $\hat{D}_l(k)$ is of Cartwright's class, and the lemma is thus proven.
\end{proof}
\begin{lemma}\label{233}
The meromorphic function $\hat{\alpha}_l(k)\equiv0$ if and only if $n(r\hat{x})\equiv1$ along the incident angle $\hat{x}$.
\end{lemma}
\begin{proof}
We begin with $\frac{{\hat{y}}'_l(\hat{r};k)}{\hat{y}_l(\hat{r};k)}\equiv k\frac{{j}'_l(k\hat{r})}{j_l(k\hat{r})}$ as a meromorphic function in $k$. Referring to \cite{Carlson,Carlson2,Carlson3,Po},
${j}'_l(k\hat{r})$ has zeros asymptotically distributed near the zeros of $\cos(k\hat{r})$; $j_l(k\hat{r})$ near the zero set of $\sin(k\hat{r})$. A similar property holds for ${\hat{y}}'_l(\hat{r};k)$ and $\hat{y}_l(\hat{r};k)$. Hence, whenever ${j}'_l(k\hat{r})$ has a zero, ${\hat{y}}'_l(\hat{r};k)$ has a zero;   $\hat{y}_l(\hat{r};k)$ has one whenever $j_l(k\hat{r})$ has. The two perturbations, $j_l(k\hat{r})$ and $\hat{y}_l(\hat{r};k)$, have the same set of Neumann and Dirichlet eigenvalue.
By the inverse spectral uniqueness of Sturm-Liouville problem \cite{Aktosun,Po}, we have $n\equiv1$. The sufficient condition is obvious. This proves the lemma.

\end{proof}
Thus, Lemma \ref{22} merely describes the eigenvalue density of the problem~(\ref{125}). To describe the density for~(\ref{226}), we may apply the translation invariant properties of interior transmission eigenvalues \cite[Lemma 1.3]{Chen6}. Alternatively, we may consider from the point of view of uniqueness theorem of ODE as in the Introduction: In $[\hat{r}_1,\hat{r}_2]$, except the previous eigenvalues of~(\ref{125}), we consider the new eigenvalue density of the problem from the second interval, that is,
\begin{eqnarray}
\left\{
  \begin{array}{ll}
    \hat{y}_l''(r;k)+(k^2n(r\hat{x})-\frac{l(l+1)}{r^2})\hat{y}_l(r;k)=0,\,\hat{r}_{1}<r<\hat{r}_{2};\vspace{8pt}\\
    \hat{D}_{l}(k;\hat{r}_{1})=0,\, \hat{D}_{l}(k;\hat{r}_{2})=0,
  \end{array}
\right.
\end{eqnarray}
is
\begin{equation}
0,
\end{equation}
because $\frac{\hat{y}_l(r;k)}{r}$ and $j_l(rk)$ satisfy the same differential equation and initial condition at $\hat{r}_1$ until $\hat{r}_2$. Thus, there is only trivial eigenfunctions in $[\hat{r}_{1},\hat{r}_{2}]$.
One can prove inductively the density in each intersection interval, so we state the following theorem.
\begin{theorem}
Let $\hat{N}(\alpha,\beta,R)$ denote the number of the eigenvalues of~(\ref{226})
inside the angle $[\alpha,\beta]$ and $|z|\leq R$ in $\mathbb{C}$. We define the
density function as
\begin{equation}
\hat{\Delta}(\alpha,\beta):=\lim_{R\rightarrow\infty}\frac{\hat{N}(\alpha,\beta,R)}{R}.
\end{equation}
Then
\begin{equation}
\hat{\Delta}(-\epsilon,\epsilon)=\frac{|\hat{r}\chi_D|+\int_{0}^{\infty}n^{\frac{1}{2}}(\rho\hat{x})\chi_Dd\rho}{\pi},
\end{equation}
in which $\chi_D$ is $1$ in $D$; zero otherwise, and $|\cdot|$ is the Lebesgue measure.
\end{theorem}
\par
We elaborate on the eigenvalue tunneling.
\begin{proposition}\label{310}
If $k$ satisfies the following ODE for some
 $l\geq0$:
 \begin{eqnarray}\label{4.1}
\left\{
  \begin{array}{ll}
    \hat{y}_l''(r)+(k^2n(r\hat{x})-\frac{l(l+1)}{r^2})\hat{y}_l(r)=0,\,\hat{r}_0:=0<r<\infty;\vspace{4pt}\\
    \hat{D}_{l}(k;\hat{r}_0)=0;\vspace{4pt}\\
\hat{D}_{l}(k;\hat{r}_1)=0,
  \end{array}
\right.
\end{eqnarray}
then the $k$
is the eigenvalue to the following system  for the same $l$ as well:
\begin{eqnarray}\label{2222}
\left\{
  \begin{array}{ll}
    \hat{y}_l''(r)+(k^2n(r\hat{x})-\frac{l(l+1)}{r^2})\hat{y}_l(r)=0,\,0<r<\infty;\vspace{4pt}\\
    \hat{D}_{l}(k;\hat{r}_{0})=0,\, \hat{D}_{l}(k;\hat{r}_{1})=0,\,\ldots,\hat{D}_{l}(k;\hat{r}_{\hat{M}})=0.
  \end{array}
\right.
\end{eqnarray}
\end{proposition}
\begin{proof}
Let the eigenvalue $k$ solve the system of~(\ref{4.1}) for some $l\geq0$.
The first two equations there give an entire function in $k$, and the third condition implies that the eigenvalues $k$ form a discrete set in $\mathbb{C}$ \cite{Chen5}.
For this $k$ given by~(\ref{4.1}), we consider the ODE system with the mixed boundary condition:
\begin{eqnarray}\label{3.35}
\left\{
  \begin{array}{ll}
    \hat{y}_l''(r)+(k^2n(r\hat{x})-\frac{l(l+1)}{r^2})\hat{y}_l(r)=0,\,\hat{r}_{1}<r<\hat{r}_{2};\vspace{4pt}\\
    \hat{D}_{l}(k;\hat{r}_{1})=0;\vspace{4pt}\\
\hat{D}_{l}(k;\hat{r}_{2})=0.
  \end{array}
\right.
\end{eqnarray}
The second equation above with $\hat{D}_{l}(k;\hat{r}_{1})=0$  gives a new initial value problem at $\hat{r}_1$, that is,
\begin{eqnarray}
\left\{
  \begin{array}{ll}
    &j_l(k r)|_{r= \hat{r}_1}=\frac{\hat{y}_l(r;k)}{r}|_{r= \hat{r}_1};\vspace{5pt}\\
    &[j_l(k r)]'|_{r= \hat{r}_1}=[\frac{\hat{y}_l(r;k)}{r}]'|_{r= \hat{r}_1}.
  \end{array}
\right.
\end{eqnarray}
Similarly, $\hat{y}_l(\hat{r};k)$ initiates at $\hat{r}_2$ again and then consecutively into the infinity with the same $k$. Moreover,
any eigenvalue $k$ of~(\ref{4.1}) satisfies $\hat{D}_{l}(k;\hat{r}_{2})=0$ in~(\ref{3.35}) by the unique continuation result~(\ref{113}) of $w_{l,m}(x;k)$ and similarly at each intersection points in $\hat{R}$. Hence, the given $k$ satisfies
\begin{eqnarray}
\hat{D}_{l}(k;\hat{r}_{0})=0,\, \hat{D}_{l}(k;\hat{r}_{1})=0,\,\ldots,\hat{D}_{l}(k;\hat{r}_{\hat{M}})=0.
\end{eqnarray}
This proves the lemma.
\end{proof}
In general, we can take $\hat{r}_j$ as the reference point.
\begin{proposition}
If $k$ satisfies the following ODE for some $l\geq0$:
 \begin{eqnarray}\label{4.6}
\left\{
  \begin{array}{ll}
    \hat{y}_l''(r)+(k^2n(r\hat{x})-\frac{l(l+1)}{r^2})\hat{y}_l(r)=0,\,\hat{r}_j<r<\hat{r}_{j+1};\vspace{4pt}\\
    \hat{D}_{l}(k;\hat{r}_j)=0;\vspace{4pt}\\
\hat{D}_{l}(k;\hat{r}_{j+1})=0,
  \end{array}
\right.
\end{eqnarray}
then the $k$
is the eigenvalue to the following system
\begin{eqnarray}\label{2277}
\left\{
  \begin{array}{ll}
    \hat{y}_l''(r)+(k^2n(r\hat{x})-\frac{l(l+1)}{r^2})\hat{y}_l(r)=0,\,\hat{r}_j<r<\infty;\vspace{4pt}\\
    \hat{D}_{l}(k;\hat{r}_j)=0,\, \hat{D}_{l}(k;\hat{r}_{j+1})=0,\,\ldots,\hat{D}_{l}(k;\hat{r}_{\hat{M}})=0.
  \end{array}
\right.
\end{eqnarray}
\end{proposition}
\section{A Proof of Theorem \ref{1.3}}
\begin{proof}
Let $k$ be an eigenvalue of~(\ref{1.1}). Thus, $k$ is an eigenvalue of~(\ref{18}) for some $l$, and then extends toward the origin along some $\hat{x}\in\mathbb{S}^2$ for any $l$. Therefore, it is an eigenvalue of~(\ref{2222}) for some $l$ along some $\hat{x}\in\mathbb{S}^2$. Conversely, if $k$ is an eigenvalues of~(\ref{2222}), then it solves~(\ref{18}) at $|x|=R_0$.  The system~(\ref{18}) is radially symmetrically defined. Thus, with this same $k$, we can solve the ODE~(\ref{2222}) along any other direction. Thus the Fourier coefficients $\{\hat{y}_l(r;k)\}$ is defined for all angle $\hat{x}$. Accordingly, $w(x;k)$ is extended in $\mathbb{R}^3$ from the far fields with the given $k$.

\par
If  $n^i$, $i=1,2,$ are two indices of refraction  with identical interior transmission eigenvalues, then
the spectrum satisfies  $\hat{D}_l(k;\hat{r})=0$ for all $\hat{r}$ and some $l$ for any $\hat{x}$. This is an one dimensional inverse Sturm-Liouville problem \cite{Aktosun,Chen,Chen3,Chen5}. In particular, we start with $j=0$. That is,~(\ref{4.1}) holds for one $l$ along some $\hat{x}\in\mathbb{S}^2$. Hence, inverse uniqueness of the Bessel operator \cite[Theorem\,1.2,\,Theorem\,1.3]{Carlson2} proves that $n^1(r\hat{x})\equiv n^2(r\hat{x})$ in  $[\hat{r}_{0},\hat{r}_{1}]$.
In this case, we have $n^1(r\hat{x})=n^2(r\hat{x})$ for $r\in[\hat{r}_{0},\hat{r}_{1}]$. The argument holds for all $j$ inductively.
This proves Theorem \ref{13}.
\end{proof}

\begin{acknowledgement}
The author wants to thank Prof. Chao-Mei Tu at NTNU for proofreading an earlier version of this manuscript and  anonymous referees for suggesting some references in value distribution theory.
\end{acknowledgement}


\begin{thebibliography}{widest-label}

\bibitem{Aktosun}T. Aktosun, D. Gintides and V. G. Papanicolaou,
The uniqueness in the inverse problem for transmission eigenvalues
for the spherically symmetric variable-speed wave equation,
Inverse Problems, V. 27, 115004 (2011).
\bibitem{Aktosun2}T. Aktosun, D. Gintides and V. G. Papanicolaou, Reconstruction of the wave speed from transmission eigenvalues for the spherically symmetric variable-speed wave equation. Inverse Problems, V. 29, no. 6, 065007 (2013).
\bibitem{Boas}R.P. Boas, Entire functions, Academic Press, New York, 1954.


\bibitem{Cakoni}F. Cakoni, D.
Colton, and H. Haddar, The interior transmission eigenvalue
problem for absorbing media, Inverse Problems, V. 28, no. 4,
045005 (2012).
\bibitem{Cakoni2}F. Cakoni, D. Colton and D. Gintides, The interior transmission eigenvalue problem, SIAM J. Math. Anal. 42, 2912--2921 (2010).

\bibitem{Carlson}R. Carlson, Inverse spectral theory for some singular Sturm-Liouville problems, Journal of Differential Equations, V.106, 121--140 (1993).
\bibitem{Carlson2}R. Carlson, A Borg-Levinson theorem for Bessel operators, Pacific Journal of Mathematics, Vol. 177, No. 1, 1--26 (1997).
\bibitem{Carlson3}R. Carlson, Inverse Sturm-Liouville problems with a singularty at zero, Inverse problems, 10, 851--864 (1994).
\bibitem{Cartwright}M. L. Cartwright, On the directions of Borel of functions which are regular
and of finite order in an angle,
 Proc. London Math. Soc. Ser. 2 Vol. 38, 503--541 (1933).
\bibitem{Cartwright2}M. L. Cartwright, Integral functions,
Cambridge University Press, Cambridge, 1956.
\bibitem{Chen}L. -H. Chen, An uniqueness result with some density theorems with interior transmission eigenvalues, Applicable Analysis, DOI:
    10.1080/00036811.2014.936403.
\bibitem{Chen3}L. -H. Chen, A uniqueness theorem on the eigenvalues of spherically symmetric interior transmission
problem in absorbing medium,  Complex Var. Elliptic Equ., 60, no. 2, 145--167 (2015).

\bibitem{Chen5}L. -H. Chen,
On the inverse spectral theory in a non-homogeneous interior transmission problem, Complex Var. Elliptic Equ., DOI: 10.1080/17476933.2014.970541.
\bibitem{Chen6}L. -H. Chen, On Certain Translation Invariant Properties of Interior Transmission Eigenvalues, Forthcoming.
\bibitem{Colton} D. Colton and P. Monk, The inverse scattering problem for time-harmonic acoustic waves in an inhomogeneous medium, Q. Jl. Mech. appl. Math. Vol. 41, 97--125 (1988).
\bibitem{Colton3}D. Colton, L. P\"{a}iv\"{a}rinta and J.
Sylvester, The interior transmission problem, Inverse problems and
imaging, V. 1, no. 1, 13--28 (2007).
\bibitem{Colton2}D. Colton and
 R. Kress, Inverse acoustic and electromagnetic scattering theory,
2rd ed. Applied mathemtical science, V. 93, Springer--Verlag, Berlin,  2013.
\bibitem{Dickson}D. G. Dickson, Zeros of exponential sums, Proc. Amer. Math. Soc, 16, 84--89 (1965).
\bibitem{Dickson2}D. G. Dickson, Expansions in series of solutions
of linear difference-differential and infinite order differential
equations with constant coefficients, Memoirs of the American
Mathematical Society, Rhode Island, USA, No. 23, 1957.


\bibitem{Kirsch86}A. Kirsch, The denseness of the far field patterns for the transmission problem, IMA J. Appl. Math, 37, no. 3, 213--225 (1986).
\bibitem{Kirsch}A. Kirsch, On the existence of transmission eigenvalues, Inverse problems and imaging, 3, 155--172 (2009).
\bibitem{Koosis}P. Koosis, The logarithmic integral I, Cambridge Univisity Press, New York, 1997.

\bibitem{L}E. Lakshtanov and B. Vainberg, Bounds on positive interior transmission eigenvalues, Inverse Problems, V. 28, No. 10, 0266--5611 (2012).
\bibitem{La}E. Lakshtanov and B. Vainberg,  Weyl type bound on positive interior transmission eigenvalues, Comm. Partial Differential Equations,  39, no. 9, 1729--1740 (2014).

\bibitem{Levin}B. Ja. Levin, Distribution of zeros of entire
functions, revised edition, Translations of mathematical
mongraphs, American mathemtical society, Providence, 1972.

\bibitem{Levin2}B. Ja. Levin, Lectures on entire functions,
Translation of mathematical monographs, V. 150, AMS, Providence, 1996.
\bibitem{Liu}H. Y. Liu, Schiffer's conjecture, interior transmission eigenvalues and invisibility cloaking: singular problem vs. nonsingular problem, Geometric analysis and integral geometry, , Contemp. Math, 598, Amer. Math. Soc., Providence, RI,  147--154(2013).
\bibitem{Mc}J. R. McLaughlin and P. L. Polyakov, On the uniqueness of a spherically symmetric speed of sound from transmission eigenvalues, Jour. Differentical Equations, 107, 351--382 (1994).
\bibitem{Olver}F. W. J. Olver, Asymptotics and special functions, Academic Press, New York, 1974.

\bibitem{Po}J. P\"{o}schel and E. Trubowitz, Inverse spectral theory, Academic Press, Orlando, 1987.

 \bibitem{Rynne}B. P. Rynne and B. D. Sleeman, The interior transmission problem and inverse scattering from inhomogeneous media, SIAM J. Math. Anal. 22, no.6, 1755--1762 (1991).
\bibitem{Sun}J. Sun, Estimation of transmission eigenvalues and the index of refraction from Cauchy data, Inverse Problems, 27, 015009 (2011).
\bibitem{Sun2}F. Zeng, T. Turner and J. Sun, Some results on electromagnetic transmission eigenvalues, DOI: 10.1002/mma.3058.
\bibitem{W}http://mathworld.wolfram.com/Tangent.html.
\end{thebibliography}
\end{document}